\documentclass[11pt]{amsart}

\usepackage{amsthm,amsmath,amsfonts, amssymb}

\def \ZZ{\mathbb Z}

\def \EE{\mathbb E}

\theoremstyle{plain}
\newtheorem{theorem}{Theorem}[section]

\newtheorem{lemma}[theorem]{Lemma}
\newtheorem{problem}[theorem]{Problem}

\theoremstyle{definition}

\begin{document}

\title[Furstenberg's filtering problem]{A note on Furstenberg's filtering problem}
\author[R.~Garbit]{Rodolphe Garbit}
\date{\today}
\address{Laboratoire de Math\'ematiques Jean Leray, UMR CNRS 6629\\
Universit\'e de Nantes\\ BP 92208\\ 44322 Nantes Cedex 3\\ France.}
\email{rodolphe.garbit@univ-nantes.fr}

\begin{abstract}
This note gives a positive answer to an old question in elementary probability theory that arose in  Furstenberg's seminal article ``Disjointness in Ergodic Theory.''
As a consequence, Furstenberg's filtering theorem holds without any integrability assumption.
\end{abstract}

\maketitle

\section{Disjointness and Filtering}
In his seminal article~\cite{Fur67}, H.~Furstenberg introduced the notion of {\em disjointness} of two stationary random processes.
A {\em joining} of two stationary processes $\{X_n\}$ and $\{Y_n\}$ is a two-dimensional stationary process $\{(X'_n,Y'_n)\}$ whose marginal distributions are 
those of $\{X_n\}$ and $\{Y_n\}$, respectively. The processes $\{X_n\}$ and $\{Y_n\}$ are said to be {\em disjoint} if any joining $\{(X'_n,Y'_n)\}$ is such that $\{X'_n\}$ and $\{Y'_n\}$ are independent.
Among other purposes this notion was used to study the following filtering problem:
if $\{X_n\}$ and $\{Y_n\}$ are two independent stationary sequences of real-valued random variables, $\{X_n\}$ being the emitted signal and $\{Y_n\}$ the noise, is it possible to recover $\{X_n\}$ from the received signal $\{X_n+Y_n\}$? More precisely, the problem was to determine whether $\{X_n\}$ is a function of the sum process $\{X_n+Y_n\}$. In that case, the sequence $\{(X_n,Y_n)\}$ is said to admit a {\em perfect filter}.
Furstenberg proved the following theorem:

\begin{theorem}[\cite{Fur67}, Theorem I.5]\label{mainthm} Let $\{X_n\}$ and $\{Y_n\}$ be two stationary sequences of integrable random variables and suppose that the two sequences are disjoint. Then $\{(X_n,Y_n)\}$ admits a perfect filter.
\end{theorem}

Furstenberg then raised the question whether the integrability requirement was essential for the conclusion of the theorem.
He noticed that the integrability stipulation may be removed if the disjointness assumption is replaced by a much stronger assumption of {\em double disjointness}. The idea of double disjointness was then exploited in~\cite{Fur95} to treat a more general filtering problem
where the received signal has the form $Z_n=f(X_n,Y_n)$. In~\cite{Bul05}, the authors deal with a $\ZZ^2$-variant of the filtering problem.
But, as far as we know, the original question about the integrability assumption was still open. It is shown here that it can be removed from the statement of Theorem~\ref{mainthm}.

\section{Probabilistic Background}

From a technical point of view, the necessity of the integrability assumption in Theorem~\ref{mainthm} is purely a probabilistic question.
In the course of the proof of the filtering theorem, Furstenberg uses the following lemma:
\begin{lemma}[\cite{Fur67}, Lemma I.3]\label{A} 
Let $U_1, U_2, V_1, V_2$ denote four integrable random variables with each of the $U_i$ independent of each $V_j$. Then $U_1+V_1=U_2+V_2$ together with $\EE[U_1]=\EE[U_2]$ implies $U_1=U_2$ and $V_1=V_2$.
\end{lemma}

Furstenberg  noted that ``It would be of interest to know if the integrability stipulation may be omitted, replacing the equality of the expectations of $U_1$ and $U_2$ by equality of their distributions,'' because a positive answer would mean that the integrability assumption in Theorem~\ref{mainthm} can also be omitted.
This remark leads to the following problem:
\begin{problem}\label{B} Let $U_1,U_2,V_1,V_2$ be four real random variables such that:
\begin{enumerate}
\item $U_1$ and $U_2$ have the same distribution;
\item For all $i,j\in \{1,2\}$,  $U_i$ and $V_j$ are independent.
\end{enumerate}
Is it true that $U_1+V_1=U_2+V_2$ implies $U_1=U_2$?
\end{problem}

In an attempt to answer this question he proposed the more general problem:
\begin{problem}\label{C} Given random variables $Z_1,Z_2,W_1,W_2$ such that $Z_1$ has the same distribution as $Z_2$ and $W_1$ has the same distribution as $W_2$, is it true that 
$Z_1+W_1\leq Z_2+W_2$
implies that
$Z_1+W_1=Z_2+W_2$?
\end{problem}
Taking $Z_1=U_1V_1$, $Z_2=U_1V_2$, $W_1=U_2V_2$ and $W_2=U_2V_1$, shows that an affirmative answer to Problem~\ref{C} would imply an
affirmative answer to Problem~\ref{B}. It is easy to see that the answer to Problem~\ref{C} is positive when the variables are integrable since the expectation of the positive variable $Z_1+W_1-Z_2-W_2$ is equal to zero. However, B.~Weiss gave in~\cite{Weiss05} a negative answer to Problem~\ref{C} by exhibiting a simple counterexample with non-integrable stable random variables for which the inequality holds without equality.

We now show that the answer to Problem~\ref{B} is positive.

\begin{lemma} Let $U_1,U_2,V_1,V_2$ be four real random variables such that:
\begin{enumerate}
\item $U_1$ and $U_2$ have the same distribution;
\item For all $i,j\in \{1,2\}$,  $U_i$ and $V_j$ are independent.
\end{enumerate}
Then $U_1+V_1=U_2+V_2$ implies $U_1=U_2$.
\end{lemma}
\begin{proof}
Fix $n\geq 1$ and let $\phi_n$ be the continuous function defined by 
$$\phi_n(x)=
\begin{cases}
n &\mbox{ if }x>n\\
x &\mbox{ if }x\in[-n,n]\\
-n &\mbox{ if }x<-n
\end{cases}$$
Set $U_{i,n}=\phi_n(U_i)$ and $V_{j,n}=\phi_n(V_j)$ for $i,j=1,2$. These variables verify:
\begin{enumerate}
\item $U_{1,n}$ and $U_{2,n}$ have the same distribution;
\item For all $i,j\in \{1,2\}$,  $U_{i,n}$ and $V_{j,n}$ are independent.
\end{enumerate}
We apply Furstenberg's argument to these truncated random  
variables.
Let $H_n=(U_{1,n}-U_{2,n})(V_{2,n}-V_{1,n})$. By linearity, independence and equality of distributions, we get
\begin{align*}
\EE[H_n]&=\EE[U_{1,n}]\EE[V_{2,n}]-\EE[U_{1,n}]\EE[V_{1,n}]-\EE[U_{2,n}]\EE[V_{2,n}]+\EE[U_{2,n}]\EE[V_{1,n}]\\
&=(\EE[U_{1,n}]-\EE[U_{2,n}])(\EE[V_{2,n}]-\EE[V_{1,n}])\\
&=0\;.
\end{align*}
Furthermore, since $\phi_n$ is a non-decreasing function and  $U_1-U_2=V_2-V_1$ a.s, we see that 
$H_n\geq 0$ a.s: if $U_1-U_2=V_2-V_1\geq 0$ then $U_{1,n}-U_{2,n}\geq 0$ and $V_{2,n}-V_{1,n}\geq 0$, hence $H_n\geq 0$; the same argument holds if
$U_1-U_2=V_2-V_1\leq 0$.

Since $H_n\geq 0$ a.s. and $\EE[H_n]=0$, we have $H_n=0$ a.s.
Now, observe that $H_n\to (U_1-U_2)(V_2-V_1)$ as $n\to\infty$. Thus, 
$$(U_1-U_2)^2=(U_1-U_2)(V_2-V_1)=0\mbox{ a.s. }$$
and the lemma is proven.
\end{proof}

As already mentioned, a consequence of this lemma is that Furstenberg's filtering theorem holds without any integrability assumption. Thus, we can formulate:
\begin{theorem}Let $\{X_n\}$ and $\{Y_n\}$ be two stationary sequences of random variables and suppose that the two sequences are disjoint. Then $\{(X_n,Y_n)\}$ admits a perfect filter.
\end{theorem}


\begin{thebibliography}{99}

\bibitem{Bul05}
W.~Bu{\l}atek, M.~Lema\'nczyk, E.~Lesigne,
\newblock {\em On the filtering problem for stationnary random $\ZZ^2$-fields},
\newblock IEEE Trans. Inform. Theory {\bf 51} (2005), 3586--3593.

\bibitem{Fur67}
H.~Furstenberg,
\newblock {\em Disjointness in ergodic theory, minimal sets, and a problem in Diophantine approximation},
\newblock  Math. Systems Theory {\bf 1} (1967), 1--49.

\bibitem{Fur95}
H.~Furstenberg, Y.~Peres, B.~Weiss,
\newblock {\em Perfect filtering and double disjointness},
\newblock Ann. Inst. H. Poincar\'e Probab. Statist. {\bf 31} (1995), 453--465.

\bibitem{Weiss05}
B.~Weiss,
\newblock {\em Some remarks on filtering and prediction of stationary processes},
\newblock Israel J. Math. {\bf 149} (2005), 345--360.





\end{thebibliography}
\end{document}